\documentclass[12pt,twoside]{amsart}

\usepackage{latexsym,amssymb,amsmath, color}

\numberwithin{equation}{section}
\newtheorem{theorem}{Theorem}[section]
\newtheorem{lemma}[theorem]{Lemma}
\newtheorem{proposition}[theorem]{Proposition}
\newtheorem{corollary}[theorem]{Corollary}

\theoremstyle{definition}
\newtheorem{definition}[theorem]{Definition}

\newtheorem{example}[theorem]{Example}

\def\lm{\textrm{lm}}

\begin{document}


\title[]{Projective Nested Cartesian Codes}

\thanks{The first and the second authors are partially supported  by  CNPq  and by FAPEMIG. 
The third author was partially supported by CONACyT and Universidad Aut\'onoma 
de Aguascalientes\\
To be published in the 
Bulletin of the Brazilian Mathematical Society, New Series. 
The final publication is available at Springer via 
	http://dx.doi.org/doi:10.1007/s00574-016-0010-z
}

\author{C\'icero Carvalho}
\address{
Faculdade de Matem\'atica\\
Universidade Federal de Uberl\^andia\\
Av. J. N. \'Avila 2121\\
38.408-902 - Uberl\^andia - MG, Brazil
}
\email{cicero@ufu.br}

\author{V.\ G.\ Lopez Neumann}
\address{
Faculdade de Matem\'atica\\
Universidade Federal de Uberl\^andia\\
Av. J. N. \'Avila 2121\\
38.408-902 - Uberl\^andia - MG, Brazil
}
\email{victor.neumann@ufu.br}

\author{Hiram H. L\'opez}
\address{
Departamento de
Matem\'aticas\\
Centro de Investigaci\'on y de Estudios
Avanzados del
IPN\\
Apartado Postal
14--740 \\
07000 Mexico City, D.F.
}
\email{hlopez@math.cinvestav.mx}

\begin{abstract} In this paper we introduce a new family of codes, called 
projective nested cartesian codes. They are obtained by the evaluation of 
homogeneous polynomials of a fixed degree on a certain subset of 
$\mathbb{P}^n(\mathbb{F}_q)$, and they may be seen as a generalization of  the 
so-called projective Reed-Muller codes. We calculate the length and the 
dimension of such codes, an upper  bound for the minimum distance and the exact 
minimum distance in a special case (which includes the projective Reed-Muller 
codes).  At the end we show some relations between the parameters of these 
codes and those of the affine cartesian codes. 
\end{abstract}

\maketitle 
\section{Introduction}\label{intro}
Let $K:=\mathbb{F}_q$ be a field with $q$ elements and let $A_0,\ldots, A_n$ be 
a collection of non-empty subsets of $K$.  Consider a {\it projective cartesian set\/}
\begin{eqnarray*}
\mathcal{X}:=\left[A_0\times A_1\times\cdots\times 
A_n\right]:=\left\{(a_0:\cdots :a_n) \; : \;  a_i\in A_i
\mbox{ for all } i\right\}\subset\mathbb{P}^{n},
\end{eqnarray*}
where $\mathbb{P}^{n}$ is a projective space over the field $K.$

In what follows $d_i$ denotes $|A_i|$, the
cardinality of $A_i$ for $i=0,\ldots,n$. We shall always 
assume that $2\leq d_i\leq d_{i+1}$ for all $i$. 

Let $S:=K[X_0,\ldots,X_n]$ 
be a polynomial ring  
over the field $K$, 
let $P_1,\ldots,P_m$ be the points of $\mathcal{X}$ written with the usual (see e.g.\ \cite{sorensen}, \cite{lachaud}, \cite{ballet}) representation for
projective points, that is, zeros to the left and the first nonzero entry equal 
1, and
let $S_{d}$ be
the $K$-vector space of all homogeneous polynomials of $S$ of degree $d$ together with the zero polynomial.
The {\it evaluation map\/}
\begin{equation*}\label{ev-map}
\varphi_d\colon S_{d}\longrightarrow K^{|\mathcal{X}|},\ \ \ \ \ 
f\mapsto \left(f(P_1),\ldots,f(P_m)\right),
\end{equation*}
defines a linear map of
$K$-vector spaces. The image of $\varphi_d$, denoted by $C_{\mathcal{X}}(d)$,
defines a linear code (as usual by a {\it linear code\/} we mean
a linear subspace of $K^{|\mathcal{X}|}$). We call
$C_{\mathcal{X}}(d)$ a {\it projective cartesian code\/}
of order $d$ defined over $A_0,\ldots, A_n$.  
Thus the projective cartesian codes are part of the family of evaluation codes 
defined on a 
subset of a projective space, see \cite{c-v}, \cite{c-duursm}, 
\cite{comp-inters} and \cite{sarabia-renteria}  for other examples.
An important special case of the projective cartesian codes, which served as 
motivation for our 
work, is the one where  $A_i = K$ for all $i = 0, \ldots, n$. Then  we have 
$\mathcal{X} = \mathbb{P}^{n}$ and $C_{\mathcal{X}}(d)$ is the so-called  
projective Reed-Muller code (of order $d$), as  defined and studied in  
\cite{lachaud} or \cite{sorensen}.

The {\it dimension\/} and the {\it length\/} of $C_{\mathcal{X}}(d)$ 
are given by $\dim_K C_{\mathcal{X}}(d)$ (dimension as $K$-vector space)
and  $|{\mathcal{X}}|$  respectively.
The {\it minimum
distance\/} of $C_{\mathcal{X}}(d)$  is given by 
$$
\delta_{\mathcal{X}}(d)=\min\{|\varphi_d(f)|
\; : \; \varphi_d(f)\neq 0; f\in S_{d}\},
$$
where $|\varphi_d(f)|$ is the number of non-zero
entries of $\varphi_d(f)$.
These are the main parameters of the code  $C_{\mathcal{X}}(d)$ and they are presented in  
the main results of this paper, although we find the minimum distance only  when the $A_i's$ satisfy certain conditions (Definition \ref{24-03-13}).

In the next section we compute the length and the dimension of $C_{\mathcal{X}}(d)$, and to do this we use some concepts of commutative algebra which we now recall. The {\it vanishing ideal\/} of $\mathcal{X}\subset\mathbb{P}^{n}$, denoted by $I(\mathcal{X})$, is the ideal of $S$
generated by the homogeneous polynomials that vanish on all points of $\mathcal{X}$.
We are interested in the algebraic invariants (degree, Hilbert function) of $I(\mathcal{X}),$
because the kernel of the evaluation map, $\varphi_d,$ is precisely $I(\mathcal{X})_d,$
where  $I(\mathcal{X})_d:=S_{d}\cap I(\mathcal{X}).$
In general, for any subset (ideal or not) $\mathcal{F}$ of $S$ we define $\mathcal{F}_d:=\mathcal{F}\cap S_d.$
The {\it Hilbert function\/} of $S/I(\mathcal{X})$ is given by 
$$H_\mathcal{X}(d):= \dim_K(S_{d}/I({\mathcal{X}})_{d}),$$
so $H_\mathcal{X}(d)$ is precisely the dimension of $C_{\mathcal{X}}(d).$

We will also need tools from Gr\"obner bases theory, which we recall briefly. 

Let $\prec$ be a monomial order defined on the set $\mathcal{M}$ of monomials 
of the polynomial ring  $S$, i.e.  
$\prec$ is a total order on $\mathcal{M}$, we have $1 \prec M$ for any monomial 
$M$, 
and if $M_1 \prec M_2$ then $M M_1 \prec M M_2$ for all $M, M_1, M_2 \in 
\mathcal{M}$. The largest monomial in a nonzero polynomial $f$ is called the 
{\em leading monomial} of $f$ and is denoted by $\lm(f)$.

\begin{definition}
Let $I$ be an ideal of $S$. A set $\{g_1, \ldots, g_s\} 
\subset I$ is a {\em Gr\"obner basis} for $I$ (with respect to $\prec$) if 
for every $f \in I$, $f \neq 0$, we have that $\lm(f)$ is a multiple of 
$\lm(g_i)$ for some $i \in \{1, \ldots, s\}$.
\end{definition}

\noindent
Gr\"obner basis were introduced in \cite{bruno} by Bruno Buchberger, who 
proved that any ideal has a Gr\"obner basis (with 
respect to a fixed monomial order). Also it is not difficult to prove that if 
$\{g_1, \ldots, g_s\}$ is 
a Gr\"obner basis for $I$ then $I = \langle g_1, \ldots, g_s \rangle$. An 
important related 
concept is that of footprint, which we now present.
\begin{definition}
The {\it footprint} (with respect to a monomial order $\prec$) of an ideal $I 
\subset S$, denoted by $\Delta(I)$,
is the set of monomials which are not leading
monomials of any polynomial in $I$. 
\end{definition}

The footprint of an ideal $I$ is closely related to a Gr\"obner basis for $I$ 
(both being defined with respect to the same monomial order in 
$\mathcal{M}$), as the 
following result shows.

\begin{proposition}\label{gb-x-di}  Let $I \subset S$ be an 
ideal and let 
$\{g_1, \ldots, g_s\}$ be a Gr\"obner basis  for $I$. Then a monomial $M$ is in 
$\Delta(I)$ if 
and only if $M$ is not a multiple of $lm(g_i)$ for all $i = 1, \ldots, s$.
\end{proposition}
\begin{proof} The ``only if" part is obvious from the definition of the 
footprint. On 
the other hand, from the definition of Gr\"obner basis we know that if 
$M$ is not a multiple of $\lm(g_i)$ for all $i = 1, \ldots, s$ then $M$ is not 
the leading monomial  of any polynomial in  $I$. 
\end{proof}

The main property of the footprint, as proved in \cite{bruno}, is that  $\{ M + 
I \; : \; M \in \Delta(I) \}$ is a basis 
for $S/I$ as a $K$-vector space.
Assume that  $I$ is an homogeneous ideal and let  $d$ be a nonnegative integer. 
It is not difficult to see that the classes of the monomials in  
$\Delta(I)_d = \{ M \in \Delta(I) \; : \; \deg(M) = d\}$ form a basis, as a 
$K$-vector space, for  $S_d/I_d$, which gives a connection between the 
footprint and the Hilbert function. We can use this to check if a set 
$G=\{g_1 , g_2 , \ldots , g_s \}$ is a Gr\"obner basis. For this we define 
$
 \Delta(G) := \left\{ M \; : \;  \text{for all } i, {\ \rm lm}(g_i) \nmid M 
\right\}, 
$
hence  $\Delta(G)_d = \{ M \in \Delta(G) \; : \; \deg(M) = d\}$.
\begin{lemma}\label{grobner_pegada}
Fix a graded monomial order in $S$. Let $I$ be a homogeneous ideal of $S$
and $G=\{g_1,g_2,\ldots, g_s \}$ a set of generators of $I$.  
The set $G$ is a Gr\"obner Basis of $I$ if and only if the Hilbert function of 
$I$
is given by
$H_I(d) = | \Delta(G)_d|$
for all $d \geq 0$.
\end{lemma}
\begin{proof}
If $G$ is a Gr\"obner basis for $I$ then, from Proposition \ref{gb-x-di} we 
get that $\Delta(I) = \Delta(G)$, hence $H_I(d) = | \Delta(G)_d|$ for all $d 
\geq 0$. On the other hand, observe that $\Delta(I) \subset \Delta(G)$, and a 
fortiori $\Delta(I)_d \subset \Delta(G)_d$ for all $d \geq 0$. If 
$|\Delta(G)_d| = H_I(d) = |\Delta(I)_d|$ then $\Delta(I) = \Delta(G)$
so $G$ is, by definition, a Gr\"obner basis for $I$.
\end{proof}

In the next section the relation between the Hilbert function and the footprint 
established in the above Lemma will be used to prove that a certain set 
$\mathcal{G}$ is a Gr\"obner basis for $\mathcal{X}$ under certain conditions 
(see 
Proposition \ref{basedeG}). Gr\"obner bases will also play an important role 
in the proof of Proposition  \ref{dist_min}, which by its turn is a key 
ingredient in the proof of the main result of Section 3, which
determines the minimum distance of a particular type of projective cartesian 
code 
defined by the product of subfields of $K$ (see Definition 
\ref{proj_nest_fields}).
We will use more than once results about affine cartesian codes, which we now 
recall.

Let $\mathcal{Y} := A_1 \times \cdots \times A_n \subset \mathbb{A}^n$, 
where $\mathbb{A}^n$ is the $n$-dimensional affine space defined over $K$. For 
a nonnegative integer $d$ write  $S_{\leq d}$ for the $K$-linear subspace of 
$K^n$ formed by the polynomials in $K[X_1, \ldots , X_n]$ of degree up to $d$ 
together with the zero polynomial. Clearly $| \mathcal{Y}| = \Pi_{i = 1}^n d_i 
=: \tilde{m}$ and let $Q_1, \ldots, Q_{\tilde{m}}$ be the points of 
$\mathcal{Y}$. Define  $\phi_d: S_{\leq d} \rightarrow K^{\tilde{m}}$ as the 
evaluation morphism $\phi_d(g) = (g(Q)_1), \ldots, g(Q_{\tilde{m}}))$.

\begin{definition}
The image $C^{*}_{\mathcal{Y}}(d)$  of $\phi_d$ is a subvector space of $K^{\tilde{m}}$ called the {\em affine cartesian code} (of order $d$) defined over the sets $A_1, \ldots, A_n$.
\end{definition}  

These codes were introduced in \cite{lopez-villa}, and also appeared independently and in a generalized form in \cite{Geil}. They are a type of affine variety code, as defined in \cite{fl}. In \cite{lopez-villa} the authors prove that we may ignore sets with just one element, and moreover may always assume that $2 \leq d_1 \leq \cdots \leq d_n$. They also completely determine the parameters of these codes, which are as follows.

\begin{theorem}\cite[Thm.\ 3.1 and Thm.\ 3.8]{lopez-villa}\label{3.1e3.8} \\
1) The dimension of $C^{*}_{\mathcal{Y}}(d)$ is $\tilde{m}$ (i.e.\ $\phi_d$ is surjective) if $d \geq \sum_{i = 1}^n (d_1 - 1)$, and for $0 \leq d < \sum_{i = 1}^n (d_1 - 1)$ we have 
\begin{align*}
   \dim(C^{*}_{\mathcal{Y}}(d)) = & \binom{n + d }{d}  -   \sum_{i = 1}^n \binom{n + d - d_i}{d - d_i} \\ & + \cdots +   
   (-1)^j \sum_{1 \leq i_1 < \cdots < i_j \leq n} \binom{n + d - d_{i_1} - \cdots -   d_{i_j}}{d - d_{i_1} - \cdots -   d_{i_j}}
   \\ &  + \cdots +   (-1)^n \binom{n + d - d_{1} - \cdots -   d_{n}}{d - d_{1} - \cdots -   d_{n}} 
\end{align*}
where we set $\binom{a}{b} = 0$ if $b < 0$. \\
2) The minimum distance $\delta^{*}_{\mathcal{Y}}(d)$ of $C^{*}_{\mathcal{Y}}(d)$ is 1, if $d \geq \sum_{i = 1}^n (d_i - 1)$, and for $0 \leq d < \sum_{i = 1}^n (d_i - 1)$ we have 
\[\delta^{*}_{\mathcal{Y}}(d) = (d_{k + 1} - \ell ) \prod_{i = k + 2}^n d_i \]
where $k$ and $\ell$ are uniquely defined by $d = \sum_{i = 1}^k (d_i - 1) + \ell$ with $0 \leq \ell < d_{k + 1} - 1$ (if $k + 1 = n$ we understand that $\prod_{i = k + 2}^n d_i = 1$, and if $d < d_1 - 1$ then we set $k = 0$ and $\ell = d$).
\end{theorem}

We will also use a result from \cite{lopez-villa} in which the authors 
determine the (homogeneous) ideal of the  set $\bar{\mathcal{Y}} := [1 \times 
A_1 \times \cdots \times A_n]$ (in what follows we use, in a cartesian product, 
$1$ to denote the set $\{1\}$ and $0$ to denote the set $\{0\}$). 

\begin{theorem}\cite[Thm.\ 2.5]{lopez-villa}\label{2.5} 
\[
I(\bar{\mathcal{Y}} ) =\langle \Pi_{a_1 \in A_1} (X_1 - a_1 X_0), \ldots, \Pi_{a_n \in A_n} (X_n - a_n X_0) \rangle
\]
\end{theorem} 

In \cite{carvalho} there are results on higher Hamming weights of affine cartesian codes, and also a proof of the minimum distance formula stated above which is simpler from the one found in \cite{lopez-villa} and uses methods similar to the ones used here. We will need a result from \cite{carvalho} which we reproduce here for the reader's convenience.

\begin{lemma}\cite[Lemma 2.1]{carvalho}\label{2.1} 
Let $ 0 < e_1 \leq \cdots \leq e_n$ and $0 \leq s \leq \sum_{i = 1}^n (e_i - 
1)$ be integers. Let $m(a_1, \ldots , a_n) = \prod_{i = 1}^n (e_i - a_i)$, 
where $0 \leq a_i < e_i$ is  an integer for all $i = 1,\ldots, n$. Then
\[ 
\min \{ m(a_1, \ldots, a_n) \; : \; a_1 + \cdots + a_n \leq s \} =  (e_{k + 1} 
- \ell) \prod_{i = k + 2}^n e_i
\]
where $k$ and $\ell$ are uniquely defined by $s = \sum_{i =1}^k (e_i -1) + \ell 
$,  with $0 \leq \ell < e_{k + 1} - 1$ (if $s < e_1 - 1$ then take $k = 0$ and 
$\ell = s$, if $k + 1 = n$ then we understand that $\prod_{i = k + 2}^n e_i = 
1$). 
\end{lemma}

%
%
%

\section{Length and dimension}\label{10-09-13}
In this section we define the projective nested cartesian codes and
compute their length and dimension.
We keep the notation and definitions used in
Section~\ref{intro}. 

For $A,B$ subsets of $K$ we write $A^{\neq 0}$ to denote the set $A \setminus\{0\}$ and we define 
$\frac{A}{B}:=\left\{\frac{a}{b} \; : \; a\in A, b\in B^{\neq 0}\right\}.$

\begin{definition}\rm\label{24-03-13}
The projective cartesian set
$\mathcal{X}=\left[A_0\times A_1\times\cdots\times A_n\right]$
is called {\it projective nested cartesian set\/} if:
\begin{itemize}
\item[\rm (i)] for all $i=0,\ldots,n$ we have $0\in A_i,$
\item[\rm (ii)] for every $i=1,\ldots,n$ we have $\frac{A_j}{A_{i-1}}\subset 
A_j$ for $j=i,\ldots, n$.
\end{itemize}
For any $d \geq 0$ the  associated linear code $C_{\mathcal{X}}(d)$ is called a
{\it projective nested cartesian code}.
\end{definition}


\begin{lemma} \label{20-05-14}
For $A,B$ subsets of $K$ with $0 \in A \cap B$ we have 
$\frac{A}{B}\subset A \Longleftrightarrow AB \subset A.$
\end{lemma}
\begin{proof}
If $B = \{ 0 \}$ then $\frac{A}{B} = \emptyset$, $A B = \{0\} \subset A$ and 
the lemma is true.  If $B \supsetneqq \{ 0 \}$ and 
$\frac{A}{B}\subset A$ then 
 taking $b \in B^{\neq 0}$ we 
have a bijection $A \rightarrow A$ given by $a\mapsto a/b$ whose inverse is  
the map $a\mapsto ab$, so that $A B \subset A$.
Conversely, if $A B \subset A$ taking $b \in B^{\neq 0}$ we 
have a bijection $A \rightarrow A$ given by $a\mapsto ab$ whose inverse is  
the map $a\mapsto a/b$, so that
$\frac{A}{B}\subset A$.
\end{proof}

From the above Lemma we get that  condition {\rm (ii)} of Definition 
\ref{24-03-13} is equivalent to the following condition:
\begin{itemize}
\item[\rm (ii')] for every $i=1,\ldots,n$ we have $A_j A_{i-1} \subset A_j$ for $j=i,\ldots, n.$
\end{itemize}

\begin{example}\label{26-03-13}
If we take $A_i=K$ for all $i=0,\ldots,n$ then the the conditions
of Definition~\ref{24-03-13} are satisfied, so $\mathbb{P}^n$ is a projective nested cartesian set
and the projective Reed-Muller codes are 
projective nested cartesian codes.
\end{example}
%

\begin{lemma}\label{15-05-13}
If $\mathcal{X}=\left[A_0\times A_1\times A_2\times\cdots\times A_n\right]$
is a projective nested cartesian set then
$$I(\mathcal{X})=\left<X_i\prod_{a_j\in A_j}\left(X_j-a_j X_i\right) \; : \; i 
< j , \; \; i,j=0,\ldots,n\right>.$$
\end{lemma}
\begin{proof}
We will make an induction on $n$. If $n=1$ then $\mathcal{X}=\left[1\times A_n\right]\cup\left\{ (0 : 1) \right\}$
and from Theorem \ref{2.5} we get 
$I(\mathcal{X})=\left<X_0\prod_{a_1\in A_1}\left(X_1-a_1X_0\right)\right>.$ Now we assume that
the result is valid for $n-1.$ Take
$C_1:=\left[1\times A_1\times A_2\times\cdots\times A_n\right],$
$C_0:=\left[A_1\times A_2\times\cdots\times A_n\right]$ and $F\in I(\mathcal{X}).$
Let $m$ be an element of $C_0$ and write
$$F=F_1X_0+F_2,$$
where $F_2\in K\left[X_1,\ldots,X_n\right].$
As $\mathcal{X}$ is a projective nested cartesian set, $\mathcal{X}=C_1\cup\left[0\times C_0\right],$
so $\left[1,m\right],\left[0,m\right]\in \mathcal{X}.$ We have $0=F(0,m)=F_2(m),$ then
$F_2\in I(C_0)$ and by induction
$$F_2\in \left<X_i\prod_{a_j\in A_j}\left(X_j-a_jX_i\right) \; : \;  i < j, \; 
\;  i,j=1,\ldots,n\right>\, .$$
We know also $0=F(1,m)=F_1(m),$ then $F_1\in I(C_1)$ and
from Theorem \ref{2.5} we get 
$$F_1\in \left<\prod_{a_i\in A_i}\left(X_i-a_iX_0\right) \; : \; 
i=1,\ldots,n\right>\, .$$
As $F=F_1X_0+F_2$ the result is true.
\end{proof}

\begin{definition}\label{variedades}
Let $\mathcal{X}=\left[A_0 \times \cdots\times A_n\right]$ be a projective
nested cartesian set.
To compute the Hilbert function of $I(\mathcal{X})$ we define
\begin{eqnarray*}
&& \mathcal{X}_{i}:=\left[A_{n-i}\times\cdots\times A_n\right], \text{ so that 
} 
I(\mathcal{X}_{i})\subset K[X_{n-i},\ldots,X_n],  \text{ for } i=0,\ldots,n 
, \text{ and } 
\\
&&
\mathcal{X}_{i}^{*}:=\left[1\times A_{n+1-i}\times\cdots\times A_n\right], 
\text{ so that } I(\mathcal{X}_{i}^\ast)\subset K[X_{n-i},\ldots,X_n],   \text{ 
for 
} i=1,\ldots,n.
\end{eqnarray*}
\end{definition}
\begin{lemma}\label{16-05-13}
For any positive integer $d, H_{\mathcal{X}_n}(d)=H_{\mathcal{X}_{n-1}}(d)+H_{\mathcal{X}_n^\ast}(d-1).$
\end{lemma}
\begin{proof}
We know that $S_d= K[X_1,\ldots,X_n]_d \bigoplus 
X_0K\left[X_0,\ldots,X_n\right]_{d-1}$. Let $f \in I(\mathcal{X}_n)_d$,
then $f = h + X_0 g$, where $h \in K[X_1,\ldots,X_n]_d$ and
$g \in K\left[X_0,\ldots,X_n\right]_{d-1}$. By definition \ref{variedades}, it 
is easy to see that $h \in I(\mathcal{X}_{n - 1})_d$ and
$g \in I(\mathcal{X}_n^{*})_{d - 1}$ and conversely, if
$h \in I(\mathcal{X}_{n - 1})_d$ and
$g \in I(\mathcal{X}_n^{*})_{d - 1}$, then
$h + X_0 g\in I(\mathcal{X}_n)_d$. Thus 
$I(\mathcal{X}_n)_d =  I(\mathcal{X}_{n - 1})_d \bigoplus X_0 
I(\mathcal{X}_n^{*})_{d - 1}$.
Then
\begin{eqnarray*}
S_d/I(\mathcal{X}_n)_d & \simeq &
K[X_1,\ldots,X_n]_d/I(\mathcal{X}_{n - 1})_d
\oplus 
X_0 K\left[X_0,\ldots,X_n\right]_{d-1}/X_0I(\mathcal{X}_n^{*})_{d - 1} \\
& \simeq & 
K[X_1,\ldots,X_n]_d/I(\mathcal{X}_{n - 1})_d
\oplus
K\left[X_0,\ldots,X_n\right]_{d-1}/I(\mathcal{X}_n^{*})_{d - 1}
\end{eqnarray*}
which completes the proof.
\end{proof}

\begin{lemma}\label{hil-fun}
Let $\mathcal{X}=\left[A_0 \times \cdots\times A_n\right]$ be a projective
nested cartesian set. The Hilbert function of $S/I(\mathcal{X})$ is given by
\begin{align*}
H_\mathcal{X}(d) = & \, 1+  \sum_{j=1}^n
\left[
\binom{j+d-1}{d-1}
-  \sum_{i=n+1-j}^n \binom{j+d-1-d_i}{d-1-d_i}
                       \right. \nonumber \\
&   + \cdots + 
(-1)^k \sum_{n+1-j \le i_1 < \cdots < i_k \le n}
\binom{j+d-1-(d_{i_1} + \cdots + d_{i_k})}{d-1-(d_{i_1} + \cdots + d_{i_k})} 
  \\ 	
&  + \cdots +
\left. (-1)^j 
 \binom{j+d-1-(d_{n+1-j} + \cdots + d_{n})}{d-1-(d_{n+1-j} + \cdots + d_{n})}
                \right].
\end{align*}
\end{lemma}
\begin{proof}
Using Lemma~\ref{16-05-13} we have
$H_\mathcal{X}(d)=H_{\mathcal{X}_0}(d)+
\displaystyle \sum_{j=1}^nH_{\mathcal{X}_j^{*}}(d-1).$
As $\mathcal{X}_0=[1],$ then $I(\mathcal{X}_0)=0$ and $H_{\mathcal{X}_0}=1.$ From Theorem \ref{3.1e3.8} (1) we get 
\begin{align*}
H_{\mathcal{X}_{j}^{*}}(d-1)  = & \,
\binom{j+d-1}{d-1}
-  \sum_{i=n+1-j}^n \binom{j+d-1-d_i}{d-1-d_i}  
\nonumber \\
&  + \cdots +
(-1)^k \sum_{n+1-j \le i_1 < \cdots < i_k \le n}
\binom{j+d-1-(d_{i_1} + \cdots + d_{i_k})}{d-1-(d_{i_1} + \cdots + d_{i_k})} 
 \\
&  + \cdots +
(-1)^j 
\binom{j+d-1-(d_{n+1-j} + \cdots + d_{n})}{d-1-(d_{n+1-j} + \cdots + d_{n})}\, .
\qedhere
\end{align*}
\end{proof}

We come to the main result of this section.
\begin{theorem}\label{08-09-13}
Let $C_{\mathcal{X}}(d)$ be a projective nested cartesian code over $A_0,\ldots,A_n.$
The length of the code is given by $m=1+\sum_{i=1}^nd_i\cdots d_n$ and its dimension by
\begin{align*}
\dim_K C_{\mathcal{X}}(d)  = & \, 1+ \sum_{j=1}^n
\left[
\binom{j+d-1}{d-1}
-  \sum_{i=n+1-j}^n \binom{j+d-1-d_i}{d-1-d_i}  
\right. \nonumber \\
&  + \cdots +
(-1)^k \sum_{n+1-j \le i_1 < \cdots < i_k \le n}
\binom{j+d-1-(d_{i_1} + \cdots + d_{i_k})}{d-1-(d_{i_1} + \cdots + d_{i_k})} 
 \\
& + \cdots +  
\left. (-1)^j 
\binom{j+d-1-(d_{n+1-j} + \cdots + d_{n})}{d-1-(d_{n+1-j} + \cdots + d_{n})}
\right].
\end{align*}
\end{theorem}
\begin{proof}
As $\mathcal{X}=\left[A_0\times A_1\times\cdots\times A_n\right]$ is a
projective nested cartesian set, then
\begin{align*}
\mathcal{X}  = & \, \left[A_0^{\neq 0}\times A_1\times A_2\times\cdots\times A_n\right]\\
& \cup \left[0\times A_1^{\neq0}\times A_2\times\cdots\times A_n\right] \\
& \qquad \qquad \qquad \vdots\\
& \cup \left[0\times0\times 0\times\cdots\times A_{n-1}^{\neq0}\times A_n\right]  \\
& \cup \left[0\times0\times 0\times\cdots\times0\times 1\right].
\end{align*}
Condition {\rm (ii)} of Definition~\ref{24-03-13} allows us
change $A_{i}^{\neq 0}$ for $1$ for all $i = 0, \ldots, n - 1$ so we get 
$| \mathcal{X} | = 1+\sum_{i=1}^nd_i\cdots d_n$. As the kernel of the evaluation map $\varphi_d$ is $S_{d}\cap I(\mathcal{X}),$
the Hilbert function of $S/I(\mathcal{X})$ agrees with the dimension
of $C_{\mathcal{X}}(d)$, so, by Lemma \ref{hil-fun} we have the dimension.
\end{proof}

From now on we choose 
 the graded lexicographic monomial order $\prec$ in $S$,
where $X_0 \prec \cdots \prec X_n$, and 
to finish this section we show that the set
$$\mathcal{G} := \left\{ X_i\prod_{a_j\in A_j}\left(X_j-a_jX_i\right) \; : \; 
i<j, \; \; i,j=0,\ldots,n \right\}$$
is a Gr\"obner basis of the ideal $I(\mathcal{X})$. In what follows $M$ denotes 
a
monomial in $S$.

\begin{lemma}\label{number_pegada}
The number of elements of $\Delta(\mathcal{G})_d$ is given by
\begin{align*}
\, & \binom{n+d}{n} - \sum_{j=1}^n  \left( \binom{n+d - d_j}{n}  - 
\binom{n - j +d - d_j}{n - j} \right)  \\ & + \cdots +
 (-1)^k \!\!\!\! \!\!\!\!\sum_{1 \le j_1< \cdots <j_k \le n} \!\!\left( 
\binom{n+ d - (d_{j_1} + \cdots + d_{j_k})}{n} - 
\binom{n - j_1 + d - (d_{j_1} + \cdots + d_{j_k})}{n - j_1} \right) \\ &   
+ \cdots +(-1)^n \binom{n+ d - (d_{1} + \cdots + d_{n} + 1)}{n}.
\end{align*}
\end{lemma}
\begin{proof}
Observe that
$\Delta(\mathcal{G}) = 
\left\{
M \; : \; X_i X_j^{d_j} \nmid M, 0 \le i< j \le n \right\}.$
For $1 \le j \le n$, we define
$
\mathcal{M}_j :=
\left\{
M \; : \; \text{ there is } i, 0 \le i< j, X_i X_j^{d_j} \mid M
\right\}.$
Then
$
\Delta(\mathcal{G}) = \mathcal{M}_S -
\left(
\bigcup_{j=1}^n \mathcal{M}_j
\right),
$
where $\mathcal{M}_S$ is the set of all monomials in $S.$
Therefore, when we count the number of monomials of degree $d$ in $\Delta(\mathcal{G})$, from the inclusion-exclusion theorem we get 
\begin{align*}
\Delta(\mathcal{G})_d  =&  | \left({\mathcal{M}_S}\right)_d | - \sum_{j=1}^n 
| {(\mathcal{M}_j)}_d |
+ \sum_{j_1<j_2} | {(\mathcal{M}_{j_1} \cap \mathcal{M}_{j_2})}_d |  
 \\ & - \cdots
  +(-1)^k \sum_{j_1<j_2< \cdots <j_k}
| (\mathcal{M}_{j_1} \cap \mathcal{M}_{j_2} \cap \cdots \cap 
\mathcal{M}_{j_k})_d |  \\
&+ \cdots  +(-1)^n | (\mathcal{M}_{1} \cap \mathcal{M}_{2} \cap \cdots \cap 
\mathcal{M}_{n})_d | \, .
\end{align*}
Clearly 
$  |(\mathcal{M}_S)_d | = \binom{n+d}{n}.$
Let $j \in \{1, \ldots, n\}$ and let $M= X_0^{\alpha_0}. \cdots . X_n^{\alpha_n} \in (\mathcal{M}_j)_d$, then there exists $i<j$, such that
$\alpha_i \ge 1$ and $\alpha_j \ge d_j$. Taking $\beta_j = \alpha_j - d_j$ and
for $k \neq j$, $\beta_k = \alpha_k$, we have that $|(\mathcal{M}_j)_d |$ is 
the number of solutions
of $\beta_0 + \cdots +\beta_n = d - d_j,$
such that  $\beta_0 + \cdots +\beta_{j-1} \ge 1$. Then $| (\mathcal{M}_j)_d |$ 
is the number of solutions of
$\beta_0 + \cdots +\beta_n = d - d_j$ minus the number of solutions of
$\beta_j + \cdots +\beta_n = d - d_j.$ This means
$$
| (\mathcal{M}_j)_d | = \binom{n+d - d_j}{n}  - 
\binom{n - j +d - d_j}{n - j}  \, .
$$
Now let  $M= X_0^{\alpha_0}. \cdots .X_n^{\alpha_n} \in \ (\mathcal{M}_{j_1} \cap \cdots \cap \mathcal{M}_{j_k})_d$,
then there exists $i<j_1$, such that
$\alpha_i \ge 1$ and $\alpha_{j_w} \ge d_{j_w}$, for $1 \le w \le k$.
Taking $\beta_{j_w} = \alpha_{j_w} - d_{j_w}$, for $1 \le w \le k$, with 
 $l \neq j_w$ and $\beta_l=\alpha_l$, we get that $|(\mathcal{M}_{j_1} \cap 
 \cdots \cap \mathcal{M}_{j_k})_d |$
is the number of solutions of
$\beta_0 + \cdots +\beta_n = d - (d_{j_1} + \cdots + d_{j_k} )$
minus the number of solutions of
$\beta_{j_1} + \cdots +\beta_n = d - (d_{j_1} + \cdots + d_{j_k})$, hence
$$
| (\mathcal{M}_{j_1} \cap \cdots \cap \mathcal{M}_{j_k})_d | = \binom{n+ d - 
(d_{j_1} + \cdots + d_{j_k})}{n}  - 
\binom{n - j_1 + d - (d_{j_1} + \cdots + d_{j_k})}{n - j_1}  \, .
$$
For $k=n$ we have
\begin{align*}
\binom{n+ d - (d_{1} + \cdots + d_{n})}{n}  -&  \binom{n - 1 + d - (d_1 + 
\cdots 
+ 	d_n)}{n - 1} \\ 
=& \binom{n+ d - (d_{1} + \cdots + d_{n} + 1)}{n}.\qedhere
\end{align*}
\end{proof}
We use the next well-known combinatorial result
to check that $H_\mathcal{X}(d) = | \Delta(\mathcal{G})_d |$ for all $d \geq 0$.
\begin{lemma}\label{combinatorial}
Let $a,b$ be non-negative integers. Then
$ \sum_{j=0}^a \binom{j+b-1}{j} = \binom{a+b}{a}. $
\end{lemma}
%
\begin{proposition}\label{basedeG}
Let $\mathcal{X}=\left[A_0\times \cdots \times A_n\right]$ be a projective nested cartesian set. The set
$\mathcal{G}  = \left\{ X_i\prod_{a_j\in A_j}\left(X_j-a_j X_i\right) \; : \; i 
< j, \; \; i,j=0,\ldots,n \right\}$
is a Gr\"obner basis for $I(\mathcal{X})$.
\end{proposition}
\begin{proof}
From Lemma \ref{grobner_pegada} we only need to compare the formulas of Lemmas \ref{hil-fun} and \ref{number_pegada}. On the formula for the Hilbert Function, we distribute the sum, use Lemma \ref{combinatorial} and compare term by term
with the formula for the footprint. The first term is
$$	
1+\sum_{j=1}^n 
\binom{j+d-1}{d-1} = \sum_{j=0}^n 
\binom{j+d-1}{j} =
\binom{n+d}{n} \, ,
$$
the second term is 
\begin{align*}
\sum_{j=1}^n \sum_{i=n+1-j}^n \binom{j+d-1-d_i}{d-1-d_i} & = 
\sum_{i=1}^n \sum_{j=n+1-i}^n \binom{j+d-1-d_i}{j}
 \\
& =  
\sum_{j=1}^n \sum_{i=n+1-j}^n \binom{i+d-1-d_j}{i}
 \\
& =  
\sum_{j=1}^n\left(  \sum_{i=0}^n \binom{i+d-d_j-1}{i} -
                    \sum_{i=0}^{n-j} \binom{i+d-d_j-1}{i}
            \right)
 \\
& =  
\sum_{j=1}^n\left(  \binom{n+d-d_j}{n} -
                    \binom{n-j+d-d_j}{n-j}
            \right) \, ,
\nonumber
\end{align*}
and the general term is
\begin{align*}
&\sum_{j=1}^n \;\; \sum_{n+1-j \le i_1 < \cdots < i_k \le n} \; 
                 \binom{j+d-1-(d_{i_1} + \cdots + d_{i_k})}{d-1-(d_{i_1} + 
                 \cdots + d_{i_k})}  \\ &=
\sum_{1 \le i_1 < \cdots < i_k \le n} \; \; \sum_{j=n+1-i_1}^n 
                 \binom{j+d-1-(d_{i_1} + \cdots + d_{i_k})}{j}  \\ &=
\sum_{1 \le i_1 < \cdots < i_k \le n}
\left(
\binom{n+d-(d_{i_1} + \cdots + d_{i_k})}{n} - 
\binom{n-i_1 +d-(d_{i_1} + \cdots + d_{i_k})}{n-i_1}
\right).
\end{align*}
Finally, for the last term, the sum on the formula for the Hilbert function has only one term, and
$$
\binom{n+ d-1 - (d_{1} + \cdots + d_{n})}{d-1 - (d_{1} + \cdots + d_{n})}
=
\binom{n+ d - (d_{1} + \cdots + d_{n} + 1)}{n},$$
which proves the Proposition.
\end{proof}


\section{Minimum Distance}\label{14-09-13}

We start this section by presenting an upper bound for the minimum distance of 
 projective nested cartesian codes. Instead of $f(X_0, \ldots, X_n)$ we write simply $f(X)$ for a polynomial in $S$.

\begin{lemma}\label{desigualdade}  If $\mathcal{X}$ is the projective nested cartesian set over $A_0,\ldots,A_n,$ then
the minimum distance of $C_{\mathcal{X}}(d)$ satisfies
$\delta_{\mathcal{X}}(d)   \le  \left( d_{k+1}-\ell \right) d_{k+2}\cdots d_n$ if
$1 \le d \le \sum\limits_{i=1}^{n}\left( d_i-1 \right),$
and $\delta_{\mathcal{X}}(d) =  1$ in otherwise,
where $0\leq k\leq n-1$ and $0 \leq \ell < d_{k+1}-1$ are the unique integers such that 
$d - 1= \sum\limits_{i=1}^{k}\left(d_i-1\right)+\ell.$
\end{lemma}
\begin{proof}
For all $i = 0, \ldots, n$ choose $a_i \in A_i$. It is easy to see that the 
polynomial 
\[
f(X) = X_0 \prod_{i=1}^n \prod_{a \in A_i}^{a\neq a_i} ( X_i - a X_0)
\]
of degree $\sum\limits_{i=1}^{n}\left(d_i-1\right) + 1$
is zero for all points of $\mathcal{X}$ except $(1:a_1: \cdots : a_n)$. Thus  
for $d > \sum\limits_{i=1}^{n}\left(d_i-1\right)$ we get 
$\delta_{\mathcal{X}}(d) = 1$.
Let $B_{k+1} \subset A_{k+1}$ be a set with $\ell$ elements. For $d - 1= \sum\limits_{i=1}^{k}\left(d_i-1\right)+\ell$, taking
\[
 f(X) = X_0
\left( \prod_{i=1}^k \prod_{a \in A_i}^{a\neq a_i} ( X_i - a X_0)
\right)
\left( \prod_{a \in B_{k+1}} ( X_{k+1} - a X_0)
\right),
\]
we obtain the desired inequality.
\end{proof}

The next example, found by R.\ Villarreal, proves  that this upper bound is not 
reached in all cases, but we will prove
that it is the true value  
of the minimum distance in some special cases, which include the projective 
Reed-Muller codes (see Theorem \ref{dist_min_final}).

\begin{example} Let $K=\mathbb{F}_{4}$ be a finite field with $4$ elements and let
	$K_0=K_1=\mathbb{F}_2, K_2=\mathbb{F}_{4}$ be subsets of $K$. Then
	$\mathcal{X}=\left[K_0\times K_1\times K_2\right]$ is a projective nested cartesian product, and
	the minimum distance of the code $C_{\mathcal{X}}(d)$ is:
	\begin{eqnarray*}
		&&\left.
		\begin{array}{c|c|c|c|c}
			d & 1 & 2 & 3 & 4 \\ \hline
			\delta_{\mathcal{X}}(d) & 8 & 4 & 3 & 1 
		\end{array}\right.
	\end{eqnarray*}
Observe that for $d=4$, we have 
$d-1 = (2-1)+2$ and for
$f=X_2( X_2^3 + X_1^3 + X_0^3 + X_0^2X_1)$, we
have $w(f) = 1 < (d_{k+1}-\ell) = (4 - 2)=2$.
\end{example}

\begin{lemma} \label{neste_with_1}
Let $\mathcal{X} = \left[A_0\times\cdots\times A_n\right]$ be a projective nested cartesian set.
For all $j = 0, \ldots, n$ let $a_j \in A_j^{\neq 0}$ and define $B_j = a_j^{-1} A_j$. Then 
$\mathcal{Y} = \left[B_0\times \cdots\times B_n\right]$ is a projective nested cartesian set
such that $1 \in B_j$, for all $j=0,\ldots , n$, and
$C_\mathcal{X} (d) =C_\mathcal{Y} (d),$
for all degree $d$.
\end{lemma}
\begin{proof}
Let $\mathcal{X}=\{ P_1 , \ldots , P_m \}$ and $\mathcal{Y} = \{ Q_1 , \ldots , Q_m \}$, where $P_i= ( x_0: \cdots: x_ n)$ and 
$Q_i= (a_0^{-1} x_0 , \ldots , a_n^{-1} x_ n)$ for all $i = 0, \ldots, n$.
Let $v\in C_\mathcal{X} (d)$, then $v=(f(P_1) : \cdots: f(P_m))$ for some $f \in S_d$.
Define $g(X_0, \ldots , X_n) = f(a_0 X_0, \ldots , a_n X_n) \in S_d$. It is easy to see that
$v = (g(Q_1) : \cdots: g(Q_m))$, so that $C_\mathcal{X} (d)  \subset C_\mathcal{Y} (d)$.
The proof of $C_\mathcal{Y} (d)  \subset C_\mathcal{X} (d)$ is similar.
\end{proof}

Thus we see that one may always assume that $1 \in A_j$, for all $j = 0, \ldots, n$. We present now the special class of projective nested cartesian set for whose associated codes we will determine the minimum distance. 

%
%

\begin{definition} \label{proj_nest_fields}
Let $K_0\subset \cdots \subset K_n$ be subfields of $K$, with $|K_i|=d_i$ for all $0 \le i \le n$. Observe that
$d_{i+1} = d_i^{r_i}$, for some $r_i \ge 1$ and $q=d_n^{r_n}$. Then
$\mathcal{X}=\left[K_0\times\cdots\times K_n\right]$ is a projective nested cartesian set which
is called a {\it projective nested product of fields.}
\end{definition}


Clearly $\mathbb{P}^n$ is a projective nested product of fields, so our results on codes defined over such sets extend the results on projective Reed-Muller codes.


\begin{definition}
For a set $\mathcal{A} \subset \mathcal{X}$ and $f \in S_d \backslash I(\mathcal{A})$, define
$$
Z_{\mathcal{A}}(f) := \left\{P \in \mathcal{A} \; : \;f(P) = 0 \right\}.
$$
In this way, for a codeword $v = (f(P_1), \ldots , f(P_m)) \neq 0,$
where $f(X) \in S_d \backslash I(\mathcal{X})_d,$
the weight of $v$ is
$| \mathcal{X} \backslash Z_{\mathcal{X}}(f) |,$
and the minimum distance of $C_\mathcal{X} (d)$ is 
$$
\delta_\mathcal{X} (d) = \text{min}\,
\left\{
| \mathcal{X} \backslash Z_{\mathcal{X}}(f) | \; : \;
f \in S_d \backslash I(\mathcal{X})_d
\right\}.$$
\end{definition}

\begin{lemma} \label{divide}
Let $f$ be an element of $S_d$ such that
for all $t \le j \le n$ we have
$
Z_{\mathcal{X}}(X_j) \subset Z_{\mathcal{X}}(f).
$
Then there exists $g_t(X)$ in $S_{d-(n-t+1)}$ such that
$f - g_t \cdot X_t \cdots X_n \in I(\mathcal{X})$.
\end{lemma}
\begin{proof}
Write $f = g_n X_n + h_n$, where 
$h_n \in K[X_0,\ldots , X_{n-1}]_d$. 
For any $P=(x_0:\cdots : x_{n-1}:0) \in \mathcal{X}$, we have $f(P)=0$. This implies that 
$h_n \in I([K_0 \times \cdots \times K_{n-1}])$, and a fortiori we have $h_n \in I(\mathcal{X})$.
By induction on $\alpha$, suppose that for some $t+1 \le \alpha \le n$ we have
$f=g_\alpha X_\alpha \cdots X_n + h_\alpha,$
where $h_\alpha \in I(\mathcal{X})$.
Write $g_\alpha  = g_{\alpha-1} X_{\alpha-1} + \tilde{h}_{\alpha-1}$,
where $\tilde{h}_{\alpha-1} \in K[X_0,\ldots , X_{\alpha-2},X_\alpha \ldots , X_n]$.
For any $P=(x_0:\cdots :x_{\alpha-2}:0:x_\alpha : \cdots : x_n) \in \mathcal{X}$, we have $f(P)=0$.
This implies
$(\tilde{h}_{\alpha-1}X_\alpha \cdots X_n) (P)=0,$
which means
$\tilde{h}_{\alpha-1} X_\alpha \cdots X_n \in 
   I([K_0 \times \cdots \times K_{\alpha-2} 
          \times K_\alpha \times  \cdots \times K_n])\subset I(\mathcal{X})$.        
We have then
\[
 f=g_{\alpha-1} X_{\alpha-1} \cdots X_n + \tilde{h}_{\alpha-1} X_\alpha \cdots 
 X_n+h_\alpha,
\]
where  $\tilde{h}_{\alpha-1} X_\alpha \cdots X_n+h_\alpha \in I(\mathcal{X})$.
By induction on $\alpha$, our result is proved. It is easy to see that
$g_t \in S_{d-(n-t+1)}$.
\end{proof}
	

\begin{proposition}\label{dist_min}  Let $\mathcal{X}$ be the projective nested product of fields over $K_0,\ldots,K_n,$
and let $f \in S_d \backslash I(\mathcal{X})$ be a 
homogeneous polynomial on $S$ of degree $d$, with
$1 \leq d \leq \sum\limits_{i=1}^{n}\left(d_i-1\right)$
if $d_1 = \cdots = d_n$, or $1 \leq d < d_{r + 1}$ 
if there exists a positive integer $r$ such that
$d_1= d_r < d_{r+1}$. Then
\begin{equation*} \label{dist_min_eq}
| \mathcal{X} \backslash Z_{\mathcal{X}}(f) | \ge \left(d_{k+1}-\ell\right)d_{k+2}\cdots d_n\, ,
\end{equation*}
where $0\leq k\leq n-1$ and $0 \leq \ell < d_{k+1}-1$ are the unique integers such that
$d-1 = \sum\limits_{i=1}^{k}\left(d_i-1\right)+\ell.$
\end{proposition}

\begin{proof}
We will make an induction on $n$. Let $n=1$ and
$f \in  S_d\backslash I(\mathcal{X})$, where $1 \leq d \leq d_1 - 1$. Then $f$
has no more than $d$ roots in $\mathbb{P}^1$ and a fortiori in
$\mathcal{X}=[K_0\times K_1]$. Thus anyway we have  
$$
|\mathcal{X} \backslash Z_\mathcal{X} (f) | \ge (d_1+1) - d = d_1 - (d-1)\, .
$$
So now we assume that the statement of the theorem holds for the product 
$[K_0\times K_1\times\cdots\times K_{n- 1}]$.
Define
\[
\mathcal{Y}_n^{*} = [1\times K_1 \times \cdots \times K_n] \textrm{\ and\ }
\mathcal{Y}_{n-1} = [0 \times K_1 \times \cdots \times K_n],\]
in particular $\mathcal{X} = \mathcal{Y}_n^{*} \cup \mathcal{Y}_{n-1}$.
Let $f \in S_d \backslash I(\mathcal{X})$ be a polynomial of degree $d$, with 
$d$ as in the statement of the proposition. 

Suppose firstly that  
$f \in I(\mathcal{Y}_n^{*})$ (so $f \notin I(\mathcal{Y}_{n-1})$). From Theorem 
\ref{2.5} (and the fact that $K_j$ is a finite field with $d_j$ elements, for 
$j = 1, \ldots, n$)  we get that $I(\mathcal{Y}_n^{*})$ is generated by  
$\tilde{\mathcal{G}} = \{X_j^{d_j}-X_j X_0^{d_j-1} \, | \, j = 1, \ldots, n \}$. Endowing $S$ with a graded-lexicographic order $\prec$ such that $X_0 \prec X_1 \prec \cdots \prec X_n$ we get that ${\rm lm}(X_j^{d_j}-X_j X_0^{d_j-1}) = X_j^{d_j}$, for all $j = 1, \ldots, n$. Thus any pair of these leading monomials are coprime, so  
$\tilde{\mathcal{G}}$ is a Gr\"obner basis for $I(\mathcal{Y}_n^{*})$, with
respect to $\prec$ (see \cite[p.\ 104]{iva}). Dividing $f$ by the elements of 
$\tilde{\mathcal{G}}$ we find
 homogeneous 
polynomials $g_j$ 
 such that 
$f(X) = \sum_{j=1}^{n} g_j(X) (X_j^{d_j}-X_j X_0^{d_j-1})$. Observe that if 
$d_1 = \cdots = d_n$ then $g_j$, if nonzero, has degree $d - d_j$ ($j = 1, 
\ldots, n$). On the other hand, if
$d_1 = d_r < d_{r+1}$, then from $d < d_{r+1}$ we have $g_j=0$ for $j>r$, 
and for $j \in \{1, \ldots, r$ we have that $g_j$, if nonzero,  is of 
has degree $d - d_j$. 
Define $g(X) := \sum_{j=1}^{n} g_j(X) X_j,$
which is a 
homogeneous  
polynomial of degree $\tilde{d} = d-d_1 + 1$. Observe that
$g\mid_{\mathcal{Y}_{n-1}} = f\mid_{\mathcal{Y}_{n-1}}$ and
since $f \notin I(\mathcal{Y}_{n-1}),$ we must have  $g \notin I(\mathcal{Y}_{n-1}),$ and
as $\tilde{d}-1 = d -1 - (d_1 - 1) = \sum\limits_{i=2}^{k}\left(d_i-1\right)+\ell$, we can apply the induction hypothesis obtaining 
$$
|\mathcal{X} \backslash Z_\mathcal{X} (f) | =
|\mathcal{Y}_{n-1} \backslash Z_{\mathcal{Y}_{n-1}} (g) |
 \ge
\left(d_{k+1}-\ell\right)d_{k+2}\cdots d_n.
$$
This proves the proposition in the case where 
$f \in I(\mathcal{Y}_n^{*})$ and $f \notin I(\mathcal{Y}_{n-1})$.

Suppose now that $f \in I(\mathcal{Y}_{n-1})$ and $f \notin 
I(\mathcal{Y}_n^{*})$, and write $f = h + X_0 g$ where $h(X) = f(0, X_1, 
\ldots, X_n)$. Since $f|_{\mathcal{Y}_{n-1}} = 0$ we have 
$h|_{\mathcal{Y}_{n-1}} = 0$ and a fortiori $h|_{\mathcal{Y}_n^{*}} = 0$ so $h \in I(\mathcal{X})$. Observe that $f|_{\mathcal{Y}_n^{*}} = g|_{\mathcal{Y}_n^{*}}$ and
clearly the number of zeros of $g$ in $\mathcal{Y}_n^{*}$ is the same of the 
number of zeros of $g(1, X_1, \ldots, X_n)$ in the  cartesian product $K_1 
\times \cdots \times K_n$. Since $\deg(g(1, X_1, \ldots, X_n)) \leq d - 1$  a 
lower bound for the number of nonzeros of $g$ in $\mathcal{Y}_n^{*}$ may be 
obtained from  Theorem \ref{3.1e3.8}, and we have
%
$$
|\mathcal{X} \backslash Z_{\mathcal{X}} (f) | = |\mathcal{Y}_n^{*} \backslash Z_{\mathcal{Y}_n^{*}} (g) |
                                             \ge
\left(d_{k+1}-\ell \right)d_{k+2}\cdots d_n  \, .
$$

Finally suppose that $f \notin I(\mathcal{Y}_n^{*})$ and
$f \notin I(\mathcal{Y}_{n-1})$.

For $k=n-1$, i.e.\ when $d=  \sum\limits_{i=1}^{n-1}\left(d_i-1\right)+\ell+1$, we have  
$$
|\mathcal{Y}_n^{*} \backslash Z_{\mathcal{Y}_n^{*}} (f) |  \ge
d_n-\ell-1  
$$
since, as above, we may consider the number of nonzero points of $f(1, X_1, \ldots, X_n)$ in $K_1 \times \cdots \times K_n$
and use Theorem \ref{3.1e3.8}.
From $f \notin I(\mathcal{Y}_{n-1})$ we get 
$$
|\mathcal{Y}_{n-1} \backslash Z_{\mathcal{Y}_{n-1}} (f) | \ge 1\, ,
$$
which implies
$$
|\mathcal{X} \backslash Z_{\mathcal{X}} (f) |  \ge
d_n-\ell  
$$
and settles the case $k = n - 1$.
We treat now the case $k < n-1$, and we start by assuming that $\ell + d_1 \le 
d_{k+1}$.

We have that $d= \sum\limits_{i=1}^{k}\left(d_i-1\right)+\ell + 1$ and
$d - 1 = \sum\limits_{i=2}^{k}\left(d_i-1\right)+\ell + d_1 - 1$, then from 
Theorem \ref{3.1e3.8} (2) we get that 
\[
|\mathcal{Y}_n^{*} \backslash Z_{\mathcal{Y}_n^{*}} (f) |   \ge 
(d_{k+1}-\ell-1)d_{k+2} \cdots d_n \, ,
\]
and from the induction hypothesis we get 
\[
|\mathcal{Y}_{n-1} \backslash Z_{\mathcal{Y}_{n-1}} (f) |   \ge 
(d_{k+1}-(\ell+d_1-1))d_{k+2} \cdots d_n \ge d_{k+2} \cdots d_n\, .
\]
Adding both inequalities we obtain the desired result. 

From now on we can assume that 
\[
f \notin I(\mathcal{Y}_n^{*}),
 f \notin I(\mathcal{Y}_{n-1}), 0\le k<n-1 \textrm{\ and\ } \ell+d_1 > d_{k+1},
\]
in particular $\ell \ge 1$.	
In what follows we generalize some methods used by S{\o}rensen \cite{sorensen} to treat projective Reed-Muller codes. Define the set of hyperplanes
$$
\Pi := \{ \pi = Z(h) \subset \mathbb{P}^n \; : \; h =a_0X_0 + \cdots + 
a_{n-1}X_{n-1} + X_n \in K_n[X]\}.
$$
For all $\pi \in \Pi$, we want to estimate $| (\pi \cap \mathcal{X}) \backslash Z_{\mathcal{X}}(f)) |$.

For each 
$h = a_0X_0 + \cdots + a_{n-1}X_{n-1} + X_n$,
define $H \colon \mathbb{P}^n \to \mathbb{P}^n$ by
$$
H(x_0,\ldots , x_n) = (x_0:\cdots : x_{n-1}:h(x_0,\ldots , x_n)) \, .
$$
It is easy to see that $H$ is a projectivity that induces a bijection of $\mathcal{X}$ and sends the plane $\pi$
to the plane $Z(X_n)$, in fact
$$
P \in \pi = Z(h) \Longleftrightarrow H(P) \in Z(X_n) \, .
$$
It is also easy to check that 
\[
f(H(X)) := f(X_0, \ldots, X_{n - 1}, a_0 X_0 + \cdots + a_{n-1}X_{n-1} + X_n)
\]
is a homogeneous polynomial of degree $d$, and 
that the inverse projectivity $H^{-1}$ is the one associated to $h^{*} = 
-a_0X_0 - \cdots - a_{n-1}X_{n-1} + X_n$.
Let $g_h(X) = f(H^{-1}(X))$, then 
the restriction of the projectivity $H$ to $Z_{\mathcal{X}}(f)$  is a 
bijection between $Z_{\mathcal{X}}(f)$ and 
$Z_{\mathcal{X}}(g_h)$ because
%
%
$$
P \in Z_{\mathcal{X}}(f)
 \Longleftrightarrow f(P)= 0 \Longleftrightarrow g_h (H(P)) = 0
\Longleftrightarrow H(P) \in Z_{\mathcal{X}}(g_h)\, ,
$$
which implies that
$H( (Z(h) \cap  \mathcal{X}) \backslash Z_{\mathcal{X}}(f) ) 
= (Z(X_n) \cap \mathcal{X}) \backslash Z_{\mathcal{X}}(g_h)$.

To proceed we consider the following cases, regarding the possibility of $Z_{\mathcal{X}}(f)$ to contain or not a set $\pi \cap \mathcal{X}$, with $\pi \in \Pi$. \newline
{\rm (a)} \label{Sorensen}Assume that $Z_{\mathcal{X}}(f)$ does not contain  any set $\pi \cap \mathcal{X}$,
where $\pi \in \Pi$, and
define the set of pairs
$$
A_f := \{
( P,\pi) \in  \left( \mathcal{X} \backslash Z_{\mathcal{X}}(f) \right) \times 
\Pi \; : \;
P \in \pi \} \, .
$$
Let $\mathcal{X}' = [K_0 \times \cdots \times K_{n-1}]$ and for every $\pi = Z(h)$ let 
\[
g'_{h}(X_0,\ldots , X_{n-1}) = g_{h}(X_0,\ldots , X_{n-1},0).
\]
Since $Z(h) \cap  \mathcal{X} \not\subset Z_{\mathcal{X}}(f) $
we have that $g'_h$ does not vanish on $\mathcal{X}'$
and is homogeneous of degree $d$.
Thus, from $| (Z(X_n) \cap \mathcal{X}) \backslash  Z_{\mathcal{X}}(g_h) | =
|  \mathcal{X}' \backslash  Z_{\mathcal{X}'}(g'_h) |$ and 
the induction hypothesis  we get that 
\begin{equation*}\label{ineq_induction}
| (Z(h) \cap  \mathcal{X}) \backslash Z_{\mathcal{X}}(f)  | \ge
\left(d_{k+1}-\ell\right)d_{k+2}\cdots d_{n-1} \, .
\end{equation*}
So for each $\pi \in \Pi$ we have at least $\left(d_{k+1}-\ell\right)d_{k+2}\cdots d_{n-1}$ points
$P$ such that $(P,\pi) \in A_f$. From $| \Pi | = d_n^n$ we have
\begin{equation} \label{ineq1}
|A_f| \ge 
\left(d_{k+1}-\ell\right)d_{k+2}\cdots d_{n-1} d_n^n \, .
\end{equation}
Let $P=(b_0 : \cdots : b_n) \in \mathcal{X} \backslash Z_{\mathcal{X}}(f)$. If
$(b_0 : \cdots : b_{n-1}) \neq 0$ then there are $d_n^{n-1}$ hyperplanes $\pi \in \Pi$ such that
$P\in \pi$. If $P=(0:\cdots : 0 : 1)$, there is no hyperplane $\pi \in \Pi$ such that $P \in \pi$, so
\begin{equation} \label{ineq2}
|A_f| \le  |  \mathcal{X} \backslash Z_{\mathcal{X}}(f) | d_n^{n-1} \, .
\end{equation}
From \eqref{ineq1} and \eqref{ineq2} we get
$$
|  \mathcal{X} \backslash Z_{\mathcal{X}}(f) | \ge
\left(d_{k+1}-\ell\right)d_{k+2}\cdots d_n \, .
$$
{\rm (b)} Assume that $Z_{\mathcal{X}}(f)$ contains a set $\pi \cap \mathcal{X}$,
for some $\pi \in \Pi$. To complete the proof we will consider two subcases. \\  \\
Subcase b.1: Assume that $d_{k+1} < d_n$. 
Applying the projectivity $H$ corresponding to $\pi$ and passing from $f(X)$ to 
$f(H^{-1}(X))$ we may assume that $\pi = Z(X_n)$. From Lemma \ref{divide}
 there exists a homogeneous polynomial $g$ of degree $d - 1$
 such that
$f - gX_n \in I(\mathcal{X})$, which means $Z_{\mathcal{X}}(f) = Z_{\mathcal{X}} (gX_n)$. For
$\widetilde{\mathcal{X}}:=[1\times K_1 \times \cdots \times K_{n-1}\times K_n^{\neq 0}]$
we have $\mathcal{Y}_n^{*} \setminus Z_{\mathcal{Y}_n^{*}}(f) = \widetilde{\mathcal{X}} \setminus Z_{\widetilde{\mathcal{X}}} (g)$.
As before we may get a lower bound for $\widetilde{\mathcal{X}} \backslash Z_{\widetilde{\mathcal{X}}} (g)$ by using Theorem \ref{3.1e3.8} to obtain a lower bound for the number of nonzero points of $g(1, X_1, \ldots, X_n)$ in 
$K_1 \times \cdots \times K_{n-1}\times K_n^{\neq 0} \in \mathbb{A}^n$. 
To do this we observe that $g(1, X_1, \ldots, X_n)$  is a polynomial of degree at most $d - 1$, and also that 
$d_1 \leq \cdots \leq d_{n - 1}$ and $d_{k + 1} \leq d_n - 1$. Thus when we write  $K_1, \ldots, K_{n - 1}, K_n^{\neq 0}$  in order of increasing size the set $K_n^{\neq 0}$ does not appear before $K_{k + 1}$. In \cite{lopez-villa} the authors prove that this reordering does not affect the lower bound in Theorem \ref{3.1e3.8} (2) so we get 
%
$$
|\widetilde{\mathcal{X}} \backslash Z_{\widetilde{\mathcal{X}}} (g) |  \ge
(d_{k+1}-\ell)d_{k+2} \cdots d_{n-1}(d_n-1) \, .
$$
On the set $\mathcal{Y}_{n-1}$ we can use the induction hypothesis, observing that $d - 1 = \sum\limits_{i=2}^{k+1}\left(d_i-1\right)+\ell + d_1 - d_{k+1}$
and $0<\ell + d_1 - d_{k+1} \le d_{k+2}-1$, so  
\begin{equation}\label{induction}
|\mathcal{Y}_{n-1} \backslash Z_{\mathcal{Y}_{n-1}} (f) |  \ge 
(d_{k+2}-(\ell + d_1 - d_{k+1}))d_{k+3} \cdots d_n \ge
(d_{k+1}-\ell)d_{k+2} \cdots d_{n-1}\, .
\end{equation}
Adding both inequalities, we obtain the desired result. \\ \\
Subcase b.2: Assume that $d_{k+1} = d_n$.
Let $t \in \{1, \ldots, k + 1\}$ be the least index such that $K_t = K_{t+1} = 
\cdots = K_n$. 
For $t \le j \le n$ let
$$
\Pi_j = \{ \pi = Z(h) \subset \mathbb{P}^n \; : \;
h =a_0X_0 + \cdots + a_{j-1}X_{j-1} + X_j + a_{j+1}X_{j+1} +\cdots + a_nX_n \in K_n[X]
\}\, .
$$

If for some $j \in \{t, \ldots, n\}$ all sets $\pi \cap \mathcal{X}$,
with $\pi \in \Pi_j$, are not contained in  $Z_{\mathcal{X}}(f)$ then  we may use an argument similar to the one used in (a) above
to obtain the desired result. In this argument we will use $\Pi_j$ instead of 
$\Pi$, $\mathcal{X}'_j = [K_0 \times \cdots \times \widehat{K_j} \times \cdots 
\times K_n]$ instead of $\mathcal{X}'$ (where $K_0 \times \cdots \times 
\widehat{K_j} \times \cdots \times K_n$ means that we omit the set $K_j$ in the 
product) and for every 
\[
h =a_0X_0 + \cdots + a_{j-1}X_{j-1} + X_j + 
a_{j+1}X_{j+1} +\cdots + a_nX_n \in K_n[X]
\] we will set 
\begin{align*}
g'_{h}&(X_0, \ldots,\widehat{X_j}, \ldots, X_n) \\  =&  f(X_0, \ldots, X_{j - 
1},
-a_0X_0 - \cdots - a_{j-1}X_{j-1} -  a_{j+1}X_{j+1}  -\cdots - a_nX_n, X_{j 
+ 
1}, \ldots, X_n)
\end{align*}
at the end we use that $| \Pi_j | = d_n^n = d_j^n$ to conclude the argument and 
prove the result. \\

If for all $t \le j \le n$ there exists $Z(h_j)=\pi_j \in \Pi_j$ such that
$
\pi_j \cap \mathcal{X} \subset Z_{\mathcal{X}}(f)
$
then let $H$ be the projectivity defined by  
$$
H(x_0,\ldots , x_n) = (x_0:\cdots : x_{t-1}:h_t(x_0,\ldots , x_n): x_{t+1}: \cdots : x_n). 
$$
As before, passing from $f(X)$ to $f(H^{-1}(X))$ we may assume that $Z(X_t) 
 \cap \mathcal{X} \subset Z_{\mathcal{X}}(f)$. If all sets $\pi \cap \mathcal{X}$,
 with $\pi \in \Pi_{t + 1}$, are not contained in  $Z_{\mathcal{X}}(f)$ then  again we may use an argument similar to the one used in (a) above to get the result. If there is some $\pi \in \Pi_{t + 1}$ such that 
$\pi \cap \mathcal{X} \subset Z_{\mathcal{X}}(f)$ then using an appropriate projectivity we may assume that $Z(X_{t + 1}) \cap \mathcal{X} \subset Z_{\mathcal{X}}(f)$ (note that $Z(X_{t}) \cap \mathcal{X} \subset Z_{\mathcal{X}}(f)$ continues to hold). Proceeding in this manner, we either get the result or we get that $Z(X_{j}) \cap \mathcal{X} \subset Z_{\mathcal{X}}(f)$ for all $j = t, \ldots, n$, which we assume from now on.
%
%
%
%
%
From Lemma \ref{divide}, there exists a homogeneous polynomial $g(X)$ of degree $d-(n-t+1)$,  
such that $f=g \cdot X_t \cdots X_n$.
From $f \notin I(\mathcal{Y}_n^{*})$ we get that $g$ is not
zero on the set
$\mathcal{A}=
[1 \times K_1 \times \cdots \times K_t^{*} \times \cdots \times K_n^{*}]$  and also that
$|  \mathcal{Y}_n^{*} \backslash Z_{\mathcal{Y}_n^{*}}(f) |=
|  \mathcal{A} \backslash Z_{\mathcal{A}}(g) |$. 
The number of nonzero points of $g$ in $\mathcal{A}$ is the same of the number 
of nonzero points of $g(1, X_1, \ldots, X_n)$ in $K_1 \times \cdots \times 
K_t^{*} \times \cdots \times K_n^{*} \in \mathbb{A}^n$. Observe that from the 
definition of $t$ we get $d_1 \leq \cdots \leq d_{t - 1} \leq d_t - 1 = \cdots 
= d_n - 1$ so we may apply Theorem \ref{3.1e3.8}, noting that $\deg( g(1, X_1, 
\ldots, X_n) )\leq d - 1 - (n-t)$. To apply that result we write
\begin{equation}\label{aa}
d - 1 - (n-t) = 
\sum\limits_{i=1}^{t-1}\left(d_i-1\right)+
\sum\limits_{i=t}^{k}\left((d_i-1)-1 \right)+
\ell -(n-k-1) = \sum\limits_{i=1}^{\alpha}\left(\tilde{d}_i-1\right)+\tilde{\ell},
\end{equation}
where $\tilde{d_i}$, $0 \le \alpha \le k$ and $\tilde{\ell}$ are defined by
$$
\tilde{d_i}= \left\{
\begin{array}{ll}
d_i          &    \textrm{if } 1 \leq  i < t,\\
d_i -1 \quad & \textrm{if } t \leq i \leq n,
\end{array}
\right.
$$
$$
0\le \tilde{\ell} = \sum\limits_{i=\alpha+1}^{k} \left(\tilde{d}_i-1\right) 
+\ell - (n-k-1) < \tilde{d}_{\alpha+1} - 1 
$$
(we note that if $t = k + 1$ then we omit the term  $\sum\limits_{i=t}^{k}\left((d_i-1)-1 \right)$ in \eqref{aa}).
With this notation, from Theorem \ref{3.1e3.8} we have
$$
|  \mathcal{A} \backslash Z_{\mathcal{A}}(g) |
\ge
(\tilde{d}_{\alpha+1}-\tilde{\ell})\tilde{d}_{\alpha+2}\cdots \tilde{d}_n \, .
$$
Let $a_{\alpha + 1} = d_{\alpha+1} - \tilde{d}_{\alpha+1} + \tilde{\ell}$ and $a_j = d_j - \tilde{d}_j$ for $j = \alpha + 2, \ldots, n - 1$, then 
$$
(\tilde{d}_{\alpha+1}-\tilde{\ell})\tilde{d}_{\alpha+2}\cdots \tilde{d}_{n-1}
= \prod_{i=\alpha+1}^{n-1} (d_i - a_i),
$$ and we have 
\begin{eqnarray}
\sum_{i=\alpha+1}^{n-1} a_i
     & = & 
     (d_{\alpha+1} - \tilde{d}_{\alpha+1}  + \tilde{\ell})
         + \sum_{i=\alpha+2}^{n-1} (d_i - \tilde{d}_i) =
          \tilde{\ell}
           + \sum_{i=\alpha+1}^{n-1} (d_i - \tilde{d}_i) \nonumber \\
     & = & \sum\limits_{i=\alpha+1}^{k} \left(\tilde{d}_i-1\right) 
     +\ell - (n-k-1)
     + \sum_{i=\alpha+1}^{k} (d_i - \tilde{d}_i) + (n-1 - k) \nonumber \\
     & = & \sum\limits_{i=\alpha+1}^{k} \left(d_i-1\right) + \ell   \nonumber.
\end{eqnarray}
Thus, from Lemma \ref{2.1}  we get
$\prod_{i=\alpha+1}^{n-1} (d_i - a_i)
\ge
(d_{k+1} - \ell) d_{k+2} \cdots d_{n-1},$
and a fortiori 
$$
|  \mathcal{A} \backslash Z_{\mathcal{A}}(g) |
\ge
(d_{k+1}-\ell)d_{k+2}\cdots d_{n-1} (d_n-1).
$$
From the induction hypothesis, and similarly
as \eqref{induction}, we have
$$
|\mathcal{Y}_{n-1} \backslash Z_{\mathcal{Y}_{n-1}} (f) | \ge
(d_{k+1}-\ell)d_{k+2} \cdots d_{n-1}\, .
$$
and adding both inequalities we obtain the desired result,
which concludes the proof of the Proposition.
\end{proof}

	
We come to the main result of this section.
\begin{theorem}\label{dist_min_final}  If $\mathcal{X}$ is the projective 
nested product of fields over $K_0,\ldots,K_n,$ then
the minimum distance of $C_{\mathcal{X}}(d)$ is $1$ if $d > 
\sum\limits_{i=1}^{n}\left(d_i-1\right)$. For 
$1 \leq d \leq \sum\limits_{i=1}^{n}\left(d_i-1\right)$,
in the case where $d_1 = \cdots = d_n$, or $1 \leq d < d_{r + 1}$ 
in the case where there exists a positive integer $r$ such that
$d_1= d_r < d_{r+1}$, we have 
$$\delta_{\mathcal{X}}(d)= \left(d_{k+1}-\ell\right)d_{k+2}\cdots d_n,$$
where $0\leq k\leq n-1$ and $0 \leq \ell < d_{k+1}-1$ are the unique integers 
such that 
$
d - 1 = \sum\limits_{i=1}^{k}\left(d_i-1\right)+\ell .
$
\end{theorem}
%
%
%
%
\begin{proof}
It is a consequence of Proposition \ref{dist_min} and Lemma \ref{desigualdade}.
\end{proof}

As a consequence of our main results we  recover the formula for the parameters of Projective Reed-Muller codes.

\begin{corollary}{\rm (\cite[Theorem~1]{sorensen}; \cite[Proposition 12]{idealapp})}\label{sore}
The Projective Reed-Muller code $PC_d(n,q)$ is an
$\left[\left| \mathbb{P}^n \right|,\dim C_{\mathbb{P}^n} (d),\delta_{\mathbb{P}^n}(d)\right]$-code where
\begin{itemize}
\item[\rm (a)] $\left| \mathbb{P}^n \right|=(q^{n+1}-1)/(q-1),$
\item[\rm (b)] $\dim C_{\mathbb{P}^n} (d)=\displaystyle \sum_{j=0}^n\sum_{k=0}^j(-1)^k\binom{j}{k}\binom{j+d-1-kq}{d-1-kq}$ and
\item[\rm (c)]
$$
\delta_{\mathbb{P}^n}(d)=\left\{\hspace{-1mm}
\begin{array}{ll}
\qquad \quad q^n &\mbox{ if } \quad 1=d,\\
\left(q-\ell\right)q^{n-k-1}&\mbox{ if }\quad 1< d\leq n(q-1),\\
\qquad \qquad 1&\mbox{ if } \quad n\left(q-1\right)< d;
\end{array}
\right.
$$
here $0\leq k\leq n-1$ and $1\leq \ell \leq d_{k+1}-1$ are the unique integers such that 
$d=1+k\left(q-1\right)+\ell$.
\end{itemize}
\end{corollary}

\begin{proof} From Example~\ref{26-03-13} and Theorems~\ref{08-09-13} and 
\ref{dist_min_final}
we have the result.
\end{proof}

Now we present a relationship between the parameters of codes defined over a projective nested product of fields and certain affine cartesian codes.

\begin{corollary}
Let $K_0,\ldots, K_n$ be subfields of $K$ such that
$\mathcal{X}=\left[K_0\times K_1\times\cdots\times K_n\right]$ is a projective nested product of fields and let $\mathcal{X}_i^{*}= K_{n+1-i}\times\cdots\times K_n \subset \mathbb{A}^i$, where $i=1\ldots,n$. Set $\mathcal{X}_0^{*}= \{ 1 \}$  
If $$C_\mathcal{X}(d) \quad \text{ is a }\quad
\left[\left|\mathcal{X} \right|,\dim C_\mathcal{X}(d),\delta_\mathcal{X}(d)\right] \text{-code}$$
and $$C_{\mathcal{X}_i^{*}}(d) \quad \text{ is a }\quad
\left[\left| \mathcal{X}_i^{*}\right|,\dim C_{\mathcal{X}_i^{*}}(d),\delta_{\mathcal{X}_i^{*}}(d)\right] \text{-code},$$ then
$$
\left| \mathcal{X}\right|=\sum_{i=0}^n\left|{\mathcal{X}_i^{*}}\right|, \qquad \dim C_{\mathcal{X}}(d)=\sum_{i=0}^n \dim C_{\mathcal{X}_i^{*}}(d-1)\qquad
\text{ and } \qquad \delta_{\mathcal{X}}(d)=\delta_{\mathcal{X}_n^{*}}(d-1),
$$
where $\mathcal{X}_0^{*}=\left[1\right]$ and $\delta_{\mathcal{X}_n^{*}}(0):=d_1\cdots d_n$, with the restriction that if there exists an integer $r$ such that
$d_1= \cdots =d_r < d_{r+1}$, then $d < d_{r+1}$.
\end{corollary}
\begin{proof}
It is a consequence of Theorems~\ref{08-09-13} and \ref{dist_min_final}
and {\cite[Corollary~3.8]{lopez-villa}}.
\end{proof}
\begin{example} Let $K=\mathbb{F}_{25}$ be a finite field with $25$ elements and let
$K_0=K_1=\mathbb{F}_5, K_2=\mathbb{F}_{25}$ be subsets of $K$. Then
$\mathcal{X}=\left[K_0\times K_1\times K_2\right]$ is a projective nested cartesian product, and
the length, the dimension and the minimum distance of the code $C_{\mathcal{X}}(d)$ are:
\begin{eqnarray*}
&&\left.
\begin{array}{c|c|c|c|c|c|c|c|c|c|c|c}
 d & 1 & 2 & 3 &4 &5 & 6 & 7 & 8 & 9 & 10 & 25\\ \hline
|\mathcal{X}|
& 151 & 151 & 151 & 151 & 151 & 151 & 151 & 151 & 151 & 151 & 151 \\ \hline
\dim C_{\mathcal{X}}(d)
& 3 & 6 & 10 & 15 & 21 & 27 & 33 & 39 & 45 & 51 & 141\\ \hline
\delta_{\mathcal{X}}(d) & 125 & 100 & 75 & 50 & 25 & 24 & 23 & 22 & 21 & 20 & 1
\end{array}\right.
\end{eqnarray*}
Observe that  for $d=25$, we have 
$d-1 = (5-1)+20$ and for
$$
f=X_2( X_2^{24} - (X_0^{24} + X_1^{24}+ 2 X_0^{4}X_1^{20} + 2 X_0^{20}X_1^{4})),
$$
we
have $w(f) = 1 < (d_{k+1}-\ell) = (25 - 20)=5$.
\end{example}

\end{document}